\documentclass[11pt]{amsart}
\usepackage[colorlinks=true, pdfstartview=FitV, linkcolor=blue, 
citecolor=blue]{hyperref}

\usepackage{amssymb,bbm,amscd}
\usepackage{graphicx}
\usepackage{a4wide}

\DeclareFontFamily{U}{mathx}{\hyphenchar\font45}
\DeclareFontShape{U}{mathx}{m}{n}{ <5> <6> <7> <8> <9> <10>
   <10.95> <12> <14.4> <17.28> <20.74> <24.88> mathx10 }{}
\DeclareSymbolFont{mathx}{U}{mathx}{m}{n}
\DeclareFontSubstitution{U}{mathx}{m}{n}
\DeclareMathAccent{\widecheck}{0}{mathx}{"71}

\theoremstyle{plain}
\newtheorem{theorem}{Theorem}[section]
\newtheorem{proposition}[theorem]{Proposition}
\newtheorem{lemma}[theorem]{Lemma}

\theoremstyle{definition}
\newtheorem{remark}[theorem]{Remark}

\numberwithin{equation}{section}
 
\newcommand{\dd}{\,\mathrm{d}}
\newcommand{\ii}{\ts\mathrm{i}}

\newcommand{\ts}{\hspace{0.5pt}}
\newcommand{\nts}{\hspace{-0.5pt}}

\DeclareMathOperator{\dens}{\mathrm{dens}}
\DeclareMathOperator{\card}{\mathrm{card}}

\newcommand{\vL}{\varLambda}

\newcommand{\cA}{\mathcal{A}}
\newcommand{\cB}{\mathcal{B}}
\newcommand{\cD}{\mathcal{D}}

\newcommand{\cI}{\mathcal{I}}
\newcommand{\cL}{\mathcal{L}}

\newcommand{\cR}{\mathcal{R}}
\newcommand{\cT}{\mathcal{T}}

\newcommand{\ZZ}{\mathbb{Z}}
\newcommand{\QQ}{\mathbb{Q}}
\newcommand{\RR}{\mathbb{R}\ts}
\newcommand{\CC}{\mathbb{C}\ts}

\newcommand{\NN}{\mathbb{N}}

\newcommand{\GL}{\mathrm{GL}}

\newcommand{\exend}{\hfill $\Diamond$}
\newcommand{\rar}{\xrightarrow{\quad}}

\newcommand{\vol}{\mathrm{vol}}
\newcommand{\sinc}{\,\mathrm{sinc}}
\newcommand{\oplam}{\mbox{\Large $\curlywedge$}}

\newcommand{\bs}[1]{\boldsymbol{#1}}
\newcommand{\defeq}{\mathrel{\mathop:}=}

\newcommand{\myfrac}[2]{\frac{\raisebox{-2pt}{$#1$}}
  {\raisebox{0.5pt}{$#2$}}}

\begin{document}

\title[Convergence of Fourier--Bohr coefficients]{Convergence of
  Fourier--Bohr coefficients\\[2mm] for regular Euclidean model sets}

\author{Michael Baake}
\address{Fakult\"{a}t f\"{u}r Mathematik,
  Universit\"{a}t Bielefeld,\newline \hspace*{\parindent}Postfach
  100131, 33501 Bielefeld, Germany}
\email{mbaake@math.uni-bielefeld.de}

\author{Alan  Haynes}
\address{Department of Mathematics, University of Houston, \newline
\hspace*{\parindent}3551 Cullen Blvd., Houston, TX 77204-3008, USA}
\email{haynes@math.uh.edu}

\begin{abstract}
  It is well known that the Fourier--Bohr coefficients of regular
  model sets exist and are uniformly converging, volume-averaged
  exponential sums. Several proofs for this statement are known, all
  of which use fairly abstract machinery. For instance, there is one
  proof that uses dynamical systems theory and another one based on
  Meyer's theory of harmonious sets. Nevertheless, since the
  coefficients can be defined in an elementary way, it would be nice
  to have an alternative proof by similarly elementary means, which is
  to say by standard estimates of exponential sums under an
  appropriate use of the Poisson summation formula. Here, we present
  such a proof for the class of regular Euclidean model sets, that is,
  model sets with Euclidean physical and internal spaces and
  topologically regular windows with almost no boundary.
\end{abstract}

\maketitle
\thispagestyle{empty}

\section{Setting and statement of result}

Given a Delone set $\vL\subset\RR^d$, let $\delta^{}_{\!\vL}$ denote
the \emph{Dirac comb}\index{Dirac~comb} of $\vL$, that is, the
translation-bounded (hence tempered) measure
\[
    \delta^{}_{\!\vL} \, \defeq \sum_{\lambda\in\vL}\delta^{}_\lambda \ts ,
\]
where each $\delta_\lambda$ is a normalised Dirac point measure
supported at $\lambda$. For $R>0$, define
\mbox{$\vL^{}_R\defeq\vL\cap B^{}_{\nts R}$}, where
$B^{}_{\nts R}=B^{}_{\nts R}(0)$ is the cube of side length $2 R$
centred at $0$. We continue to use the notation $B^{}_{\nts R}$, as
this is a special case of the sup-norm ball of radius $R$, namely
\begin{equation}\label{BH-eq:balls}
     B^{}_{\nts R}(y^{}_{0}) \, \defeq \,
     \bigl\{y\in\RR^{d} : \| y-y^{}_{0} \|^{}_{\infty} \leqslant R \bigr\}
\end{equation}
with $\|\cdot\|^{}_{\infty}$ denoting the sup-norm. Clearly,
$\vol \bigl(B^{}_{\nts R} (y^{}_{0}) \bigr) = (2 R)^d$.

Next, let $\widehat{\delta^{}_{\!\vL_R}}$ denote the Fourier transform
of the finite measure $\delta^{}_{\!\vL_R}$. For any test function
$f$, a simple calculation reveals that
\[
  \widehat{\delta^{}_{\!\vL_R}}(f) \, = \int_{\RR^d}  f(t)
  \!\sum_{\lambda\in\vL^{}_{\nts R}}\! e(-t\lambda)\, \dd t \ts ,
\]
where $t\lambda = \langle t \ts | \ts \lambda \rangle$ is the standard
inner product of $t$ and $\lambda$ in $\RR^d$ and
\[
  e(y) \, \defeq \, \exp (2\pi \ii y)
  \quad \text{for } y\in\CC.
\]
This shows that $\widehat{\delta^{}_{\!\vL_R}}$ is the absolutely
continuous tempered measure with Radon--Nikodym density (or
derivative) $\vol(B^{}_{\nts R}) \, a^{}_{\nts R}$, where
\begin{equation}\label{BH-eq:FourBohrR}
    a^{}_{\nts R} ( t ) \, = \, \myfrac{1}{\vol(B^{}_{\nts R})}
    \sum_{\lambda\in\vL^{}_{\nts R}} e(-t\lambda).
\end{equation}
For each $t\in\RR^d$, let
\begin{equation}\label{BH-eq:FourBohr}
    a(t) \, \defeq \lim_{R\to\infty} a^{}_{\nts R} (t) \ts ,
\end{equation}
provided the limit exists. The complex numbers $a(t)$ are called the
natural \emph{Fourier--Bohr
  coefficients}\index{Fourier--Bohr~coefficient} associated to $\vL$,
or FB coefficients for short. Here, the term `natural' refers to the
use of balls for the volume averaging, and it is worth noting that the
type of ball will not make any difference for the existence of the
limit in the cases we consider below.

It is known (for instance by work of Hof \cite{BH-Hof} and
Lenz \cite{BH-Daniel}, see also \cite{BH-TAO1}) that, when $\vL$ is a
sufficiently nice model set, the FB coefficients exist and satisfy
the important relation
\begin{equation}\label{BH-eq:Inten1}
   \bigl| a(t) \bigr|^2 \ts = \, \widehat{\gamma}
   \bigl( \{t\} \bigr) ,
\end{equation}
where $\widehat{\gamma}$ is the \emph{diffraction
  measure}\index{diffraction} associated to $\vL$. A systematic
exposition of diffraction theory is provided in \cite[Chs.~8 and
9]{BH-TAO1}. For our purposes here, it is sufficient to know that, for
a regular model set with window $W$ and cut and project scheme
$(\RR^d,\RR^{k-d},\cL)$, the diffraction measure is a pure point
measure on $\RR^d$ supported on the Fourier module
$L^{\circledast} \defeq \pi (\cL^{*}) \subset \RR^d$ of the cut and
project scheme, with
\begin{equation}\label{BH-eq:Inten2}
  \widehat{\gamma}\bigl( \{t\} \bigr) \, = \,
  \biggl| \frac{\dens(\vL)\, \widehat{\bs{1}^{}_{W}}(-t^{\star})}
    {\vol(W)}\biggr|^2 \quad \text{for}~ t\in L^{\circledast},
\end{equation}
and $ \widehat{\gamma}\bigl( \{t\} \bigr) = 0$ otherwise. Here,
$\cL^{*}$ is the dual lattice of $\cL$ and $\pi$ is the canonical
projection to $\RR^d$, as detailed below in \eqref{BH-eq:CPS}. Note
also that
\[
    \dens (\vL) \, = \, \dens (\cL) \ts \vol (W) \ts .
\]
Relevant details of definitions related to cut and project schemes and
regular model sets are provided for the reader's convenience in the
next section, where we refer to \cite{BH-TAO1} for general background, 
notation and results.

As a physical interpretation of
Eqs.~\eqref{BH-eq:FourBohrR}--\eqref{BH-eq:Inten2}, the FB
coefficients $a(t)$ give the complex amplitudes of the diffraction
image produced from the collection of equal point scatterers with
locations described by $\vL$, and the graph of intensities of this
image is the squared modulus of the FB coefficients, at the points
of $L^{\circledast}$, while the intensities vanish everywhere else.

In this paper, we give a new proof of these properties and
connections for the case of regular Euclidean model sets, using
harmonic analysis and bounds for exponential sums. Our main result is
the following theorem.

\begin{theorem}\label{BH-THM.CONVERGENCE1}
  Let\/ $\vL\subset\RR^d$ be a regular Euclidean model set for the cut
  and project scheme\/ $(\RR^d, \RR^{k-d}, \cL)$, with lattice\/
  $\cL\subset \RR^k$ and window\/ $W\subset \RR^{k-d}$.  Then, its FB
  coefficients\/ $a(t)$ exist for all\/ $t\in\RR^d$, and satisfy
\[
  a(t) \, = \, \dens(\cL) \, 
  \widehat{\bs{1}^{}_{W}}(-t^{\star}) \, = \, \dens (\vL)\,
  \frac{ \widehat{\bs{1}^{}_{W}}(-t^{\star})}{\vol (W)}
  \qquad \text{ for all } \, t\in L^{\circledast},
\]  
  while they vanish for all other\/ $t$.
\end{theorem}

To prove these results, we begin with a more precise summary of
notation and basic results.

\section{Notation and concepts}

Our general reference is \cite{BH-TAO1}. As above, for $x,y\in\RR^s$,
we use $x\ts y$ for the standard inner product. For any $y\in\CC$, we
set $e(y)\defeq\exp(2\pi\ii y)$. If $D\subseteq\RR^s$, we write
$\bs{1}^{}_{\nts D}\colon \RR^s\rar\{0,1\}$ for the indicator function
of $D$, and $\vol(D)$ for the $s$-dimensional Lebesgue measure of $D$,
when it is well defined. For the remainder of the paper, we shall
make a notational simplification and denote the sup-norm of
$x\in\RR^s$ by $|x|=\max\{|x_i|: 1\leqslant i\leqslant s\}$. In view
of the underlying Cartesian product structure, all our balls will be
defined using the sup-norm, as in \eqref{BH-eq:balls}.

A \emph{lattice}\index{lattice} $\cL\subset \RR^s$ is a co-compact,
discrete subgroup of $\RR^s$, so $\RR^s/\cL$ has a representation as a
relatively compact (hence measurable) fundamental domain of $\cL$
within $\RR^s$, which is often derived from its Voronoi or its Delone
cell by removing part of the boundary. The density\index{density} of
$\cL$, denoted by $\dens(\cL)$, exists uniformly and is the reciprocal
of the volume of such a fundamental domain
\cite[Ex.~2.6]{BH-TAO1}. The \emph{dual lattice}\index{dual~lattice}
of $\cL$ is the lattice $\cL^*\subset \RR^s$ given by
\[
    \cL^* \, = \, \bigl\{y\in \RR^s : y\ts x\in\ZZ
    ~\text{for all}~x\in\cL\bigr\} \ts .
\]
It is not difficult to show that, if $A$ is an invertible, real
$s\ts {\times}\ts s$ matrix with $\cL=A\ZZ^s$, then
$\cL^\ast=A^{*}\ZZ^s$, where $A^{*} = (A^{-1})^{T}$ is the dual
matrix.\index{dual~matrix} Indeed, if we use the columns of $A$ as the
lattice basis vectors of $\cL$, the columns of $A^{*}$ form the
corresponding dual basis. Therefore, the density of the dual lattice
is given by $\dens(\cL^{*})=1/\dens(\cL)$; see \cite[Ex.~3.1]{BH-TAO1}
for more.

Below, we frequently use various notions from asymptotic analysis as
follows.\footnote{The reader is referred to the \textsc{WikipediA}
  entry on the `big O notation' for further details and references.}
For functions $f,g\colon \RR^s\rar\CC$ and a subset $D\subseteq\RR^s$,
we write
\[
    f(x) \, = \, O \bigl( g(x) \bigr) \qquad \text{for $x\in D$}
\]
if there exists a constant $C>0$ such that
\[
     \lvert f(x) \rvert \, \leqslant \,  C \ts \lvert g(x) \rvert
     \qquad\text{for all}~x\in D\ts .
\]
If the domain $D$ is not specified, it is assumed to be $\RR^s$. We
also use  
\[
   f(x)\,\ll\, g(x)\qquad \text{for $x\in D$}
\]
to mean that $f(x)=O\bigl( g(x) \bigr)$ for $x\in D$. Whenever the
argument is implicitly clear, we will simply write $f\ll g$. Finally,
we use
\[
    f(x) \, = \, o \bigl( g(x) \bigr)
\]
to mean that
\[
   \lim_{|x|\to\infty}
   \frac{f(x)}{g(x)} 
   \, = \, 0 \ts .
\]
Let us now turn to the setting of cut and project sets. For general
background in our present context, we refer to \cite{BH-TAO1}.

\subsection{Regular model sets}\label{BH-sec:RegModSets}

First, we describe the notation which we will use for regular
(Euclidean) model sets. The embedding (or total)
space\index{space!embedding} is
$\RR^k=\RR^d\ts {\times} \ts\ts \RR^{k-d}$, with projections
$\pi\colon \RR^k\rar\RR^d$ onto the first $d$ coordinates (the
physical space),\index{space!physical} and
$\pi^{}_{\mathrm{int}}\colon \RR^k\rar\RR^{k-d}$ onto the last $(k-d)$
coordinates (the internal\index{space!internal} space). We denote the
physical space and internal space by $G$ and $H$,
respectively. Further, suppose that $\cL\subset\RR^k$ is a lattice
with the property that the map $\pi|^{}_{\cL}$ sending points of $\cL$
to $L\defeq \pi(\cL)$ is injective, and that the set
$L^{\star}\defeq\pi^{}_{\mathrm{int}}(\cL)$ is dense in $H$. This
setting is usually summarised in the form of a \emph{cut and project
  scheme}\index{cut~and~project~scheme} (CPS) as follows,
\begin{equation}\label{BH-eq:CPS}
\renewcommand{\arraystretch}{1.2}\begin{array}{r@{}ccccc@{}l}
   & G & \xleftarrow{\,\;\;\pi\;\;\,} & G \times H & 
        \xrightarrow{\;\pi^{}_{\mathrm{int}\;}} & H & \\
   & \cup & & \cup & & \cup & \hspace*{-1ex} 
   \raisebox{1pt}{\text{\footnotesize dense}} \\
   & \pi(\cL) & \xleftarrow{\; 1-1 \;} & \cL & 
   \xrightarrow{\; \hphantom{1-1} \;} & \pi^{}_{\mathrm{int}}(\cL) & \\
   & \| & & & & \| & \\
   & L & \multicolumn{3}{c}{\xrightarrow{\qquad\qquad\;\;\,\star\,
       \;\;\qquad\qquad}} 
       &  {L_{}}^{\star\nts} & \\
\end{array}\renewcommand{\arraystretch}{1}
\end{equation}
where $\star\colon L\rar L^{\star}$, defined by
\begin{equation}\label{BH-eq:stardef}
   x\, \mapsto\, x^{\star} \, \defeq \, \pi^{}_{\mathrm{int}}
   \bigl( \pi^{-1}(x)\cap\cL \bigr) ,
\end{equation}
is the corresponding \emph{star map}\index{star~map} of the CPS; see
\cite{BH-Bob} or \cite[Ch.~7]{BH-TAO1} for background.  A CPS as in
\eqref{BH-eq:CPS} is abbreviated by the triple $(G,H,\cL)$.

\begin{remark}\label{BH-rem:lattice-sum}
  The star map is originally defined on $L$, and can consistently be
  extended to its $\QQ$-span, $\QQ L$, but not to all of $G=\RR^d$.
  However, when viewing the decomposition
  $\RR^k=\RR^d\ts {\times}\ts\ts \RR^{k-d}$, one can always uniquely
  write any element $w\in\RR^k$ as $w=(x,x^{\star})$ with the
  coordinates $x=\pi(w)\in\RR^d$ and
  $x^{\star}=\pi^{}_{\mathrm{int}}(w)\in\RR^{k-d}$. Though this
  notation amounts to a double use of the symbol $\star$, the two
  points of view are consistent whenever $x\in\QQ L$.

  In particular, given the lattice $\cL$ from the CPS
  \eqref{BH-eq:CPS}, there is a canonical bijection (induced by the
  projection $\pi$) between elements $\lambda\in L$ and lattice points
  $(\lambda ,\lambda^{\nts\star})\in\cL$. Consequently, any summation
  over all points of $\cL$ can also be written as a summation over all
  elements of $L$. The corresponding property holds for the dual
  lattice $\cL^{\ast}$ as well.  Below, we will make use of this type
  of bijection repeatedly.  \exend
\end{remark}

At this point, given a set $W\subseteq H$, which is referred to as the
\emph{window}\index{window}, the cut and project set in the CPS
$(G,H,\cL)$ associated to $W$ is the point set 
\begin{equation}\label{BH-eq:oplam}
    \vL \, = \, \oplam(W)\, \defeq \, 
    \{x\in L : x^{\star}\in W \} \ts .
\end{equation}
When the window $W$ is a bounded subset of $H$ with non-empty
interior, $\oplam(W)$ is called a \emph{model set}.\index{model~set} A
model set is called \emph{regular} if $W$ is topologically regular
(meaning that it is the closure of its interior) and if its boundary,
$\partial W$, has Lebesgue measure $0$ in $H=\RR^{k-d}$; see
\cite[Sec.~7.2]{BH-TAO1} for background. Since both $G$ and $H$ are
Euclidean spaces, we call the resulting model sets (as well as the
CPS itself) \emph{Euclidean}.\index{model~set!Euclidean}

In what follows, we need the following simple property of regular
model sets. The result appears in \cite{BH-Bob-old,BH-Bob} and, strictly
speaking, is a statement about the dual CPS. Since our groups $G$ and
$H$ are both self-dual, and since we make no further use of this
duality notion below, we suppress this detail, but include a short
proof for convenience.

\begin{proposition}\label{BH-prop.DualInj}
  If\/ $(G,H,\cL)$ is a CPS as above in \eqref{BH-eq:CPS}, the
  restriction\/ $\pi|^{}_{\text{$\cL^*$}}$ of\/ $\pi$ to the dual
  lattice\/ $\cL^{*}$ is injective.
\end{proposition}

\begin{proof}
  The map $\pi|^{}_{\text{$\cL^*$}}\colon \cL^{*}\rar G$ is a
  homomorphism of additive groups. It thus suffices to show that its
  kernel is trivial. Suppose that $y\in\cL^*$ has $\pi(y)=0$, where
  $y=(\pi(y), \pi^{}_{\mathrm{int}}(y))$. Then, for any $x\in\cL$, the
  mutual orthogonality of $G$ and $H$ implies that
\[
   y\ts x\, =\, \pi(y)\ts \pi(x) + 
   \pi^{}_{\mathrm{int}}(y)\ts \pi^{}_{\mathrm{int}}(x) 
   \, = \, \pi^{}_{\mathrm{int}}(y)\ts \pi^{}_{\mathrm{int}}(x)
   \, \in \, \ZZ  \ts .
\]
As $\pi^{}_{\mathrm{int}}(\cL)$ is dense in $H$ by assumption, the
inclusion $y \ts x \in \ZZ$ for all $x\in\cL$ implies
$\pi^{}_{\mathrm{int}}(y)=0$, hence $y=0$, and the map
$\pi|^{}_{\text{$\cL^*$}}$ is injective.
\end{proof}

Next, we will establish the connection with Fourier analysis, where
we need the so-called \emph{Fourier module}\index{Fourier~module}
associated to $(G,H,\cL)$. The latter is the set
$L^{\circledast}=\pi(\cL^{*})\subset G$.

\subsection{Fourier analysis}

Let us recall some basic notions and results from Fourier analysis. In
line with \cite[Ch.~8]{BH-TAO1}, we define the \emph{Fourier
  transform}\index{Fourier~transform} of a complex-valued function
$\phi\in L^1(\RR^s)$ to be the function
$\widehat{\phi}\colon\RR^s\rar\CC$ given by
\[
   \widehat{\phi}(y) \, =
   \int_{\RR^s_{\vphantom{t}}}e(-x\ts y) \, \phi(x ) \dd x \ts ,
\]
which is continuous.  When $t\in\RR^s$ and if
$\alpha^{}_t \ts \phi \colon \RR^s\rar\CC$ denotes the function
defined by
$x \mapsto \bigl(\alpha^{}_t \ts \phi \bigr)(x)=e(-t \ts x) \ts \phi
(x)$, one clearly has $\alpha^{}_t \ts \phi\in L^1(\RR^s)$ and
\[
    \widehat{\alpha^{\rule{0pt}{1.5pt}}_t\ts\phi}\ts (y) \, = \,
    \widehat{\phi}(y+t)\qquad\text{for all}~y\in\RR^s.
\]
Suppose that $u>0$ and consider
$\phi=\bs{1}^{}_{[-\frac{u}{2},\frac{u}{2}]}\colon \RR\rar\RR$. Then,
we find
\[
  \widehat{\phi}(y)\, = \, u\sinc(\pi u \ts y)\ts ,
\]
where $\sinc(x)\defeq \sin(x)/x$ with $\sinc(0)\defeq 1$; compare
\cite[Ex.~8.3]{BH-TAO1}.  Consequently, if $D\subset \RR^s$ is a
rectangular box with faces parallel to coordinate hyperplanes (such a
box will be called \emph{aligned}\index{aligned~box} in what follows),
centred at the origin, and with side lengths $u^{}_1,\ldots ,u^{}_s$,
one has
\begin{equation}\label{BH-eq:FourCoeffs1}
  \widehat{\bs{1}^{}_{\nts D}}(y^{}_1,\ldots ,y^{}_s) \, = \ts
  \prod_{i=1}^{s} u^{}_{i} \sinc(\pi u^{}_{i}\ts y^{}_{i})\ts .
\end{equation}
Since $\lvert \sinc (x)\rvert \leqslant 1$ for $x\in\RR$, it follows
that
\begin{equation}\label{BH-eq:FourCoeffs2}
  \left|\widehat{\bs{1}^{}_{\nts D}}(y^{}_1,\ldots ,y^{}_s)\right|
  \, \leqslant \ts \prod_{i=1}^{s} \min
  \Bigl( u_i,\myfrac{1}{\pi|y_i|} \Bigr) .
\end{equation}

For $\phi\in L^{1}(\RR^s)$, the \emph{inverse}
Fourier\index{Fourier~transform} transform
$\widecheck{\phi}\colon \RR^s\rar\CC$ is defined by
\[
    y\, \mapsto\,  \widecheck{\phi}(y) \, \defeq
    \int_{\RR^s_{\vphantom{t}}} e(x \ts y) \, \phi(x) \dd x \ts .
\]
The Fourier inversion formula states that, if
$\phi,\widehat{\phi}\in L^1(\RR^s)$, one has
\[
    \phi(y) \, = \,
    \widecheck{\widehat{\phi}\,}\! (y)
    \qquad\text{for a.e.}~y\in\RR^s \nts .
\]
If $\phi$ is also continuous, this equality holds for all
$y\in\RR^s$.

Let $\phi^{}_1,\phi^{}_2\in L^1(\RR^s)$.  Then, the
\emph{convolution}\index{convolution} of $\phi^{}_1$ and $\phi^{}_2$
is the function $\phi^{}_1 \nts \ast \phi^{}_2\in L^1(\RR^s)$ defined
by
\[
    \bigl( \phi^{}_1 \nts \ast \phi^{}_2 \bigr) (y) \, =
    \int_{\RR^s_{\vphantom{t}}}\phi^{}_1(x)\,
    \phi^{}_2(y-x) \dd x \ts .
\]
Convolution is commutative, and the \emph{convolution
theorem}\index{convolution~theorem} states that
\begin{equation}\label{BH-eq:ConvFourTrans}
  \widehat{\phi^{\rule{0pt}{2pt}}_1\!\ast\nts\nts 
  \phi^{}_2}\ts (y) \, = \,
  \widehat{\phi^{}_1}(y)\,\widehat{\phi^{}_2}(y)
  \qquad\text{for all}~y\in\RR^s .
\end{equation}

Finally, we will need the following version of the \emph{Poisson
  summation formula} (PSF),\index{Poisson~summation~formula} which is
a variant of the form proved in \cite[Cor.~VII.2.6]{BH-SteiWeis}; see
\cite[Sec.~9.2]{BH-TAO1} for background.

\begin{proposition}\label{BH-prop.PSF}
  Let\/ $\cL\subset\RR^s$ be a lattice, with dual lattice\/
  $\cL^*$. Suppose that\/ $\phi$ is a continuous function with compact
  support and that
\begin{equation}\label{BH-eq:PSF2}
  \sum_{\xi\in\cL^\ast_{\vphantom{t}}} \bigl| \widehat{\phi}(\xi) \bigr|
  \, < \, \infty \ts .
\end{equation}
Then, for all\/ $y\in\RR^s$, we have the identity
\begin{equation*} 
	\sum_{\ell\in\cL}\phi(y+\ell) \, = \, 
	\dens (\cL) \!\sum_{\xi\in\cL^*_{\vphantom{t}}}\!
	\widehat{\phi}(\xi) \, e(\xi\ts y)\ts .
\end{equation*}
\end{proposition}

\begin{proof}
  Let $A \in \GL (s,\RR)$ be a fixed matrix such that $\cL=A\ts\ZZ^s$,
  and let $f^{}_{\! A} \colon \RR^s \rar \RR^s$ be defined by
  $f^{}_{\! A}(y)=\phi(Ay)$. The proof of
  \cite[Cor.~VII.2.6]{BH-SteiWeis} (which relies on slightly different
  hypotheses) guarantees that, for all $y\in\RR^s$,
\begin{equation}\label{BH-eq:PSF1}
  \sum_{n\in\ZZ^s}f^{}_{\! A} (y+n) \, =
  \sum_{n\in\ZZ^s}\widehat{f^{}_{\! A}}(n)\, e(n\ts y) \ts .
\end{equation}
Relevant to this, and to what follows, is the observation that
\begin{align*}
  \widehat{f^{}_{\! A}}(n) \,
  & = \int_{\RR^s}e(-x\ts n) \, \phi(Ax) \dd x 
  \, = \, \left| \det \left(A^{-1}\right)\right|
    \int_{\RR^s}e \bigl( - (A^{-1}t )n \bigr)
    \ts \phi(t) \dd t\\[1mm]
  & = \, \dens (\cL) \int_{\RR^s}
    e \bigl( -t \ts (A^{*}n ) \bigr) \ts \phi(t) \dd t
    \, = \, \dens (\cL)\, \widehat{\phi} (A^{*}n ) \ts ,
\end{align*}
where $A^* = (A^{-1})^T$ is the dual matrix introduced earlier.  Now,
substituting $y'=Ay$ in \eqref{BH-eq:PSF1} we get that, for all
$y'\in\RR^s$,
\begin{align*}
  \sum_{\ell\in\cL}\phi(y'+\ell) \,
  & = \, \dens (\cL)\sum_{n\in\ZZ^s}\widehat{\phi}
    (A^{*}n ) \, e \left(n \ts (A^{-1}y' ) \right) \\[1mm]
  & = \, \dens (\cL) \sum_{n\in\ZZ^s}\widehat{\phi} (A^{*}n)
    \, e \bigl( (A^{*}n ) \ts y' \ts \bigr)  
  \, = \, \dens (\cL) \!\sum_{\xi \in\cL^*_{\vphantom{t}}}\!
  \widehat{\phi}(\xi) \, e(\xi\ts y') \ts . \qedhere
\end{align*}
\end{proof}

We are now ready to approach our general result.

\section{Setup for the 
proof of Theorem \protect\ref{BH-THM.CONVERGENCE1}}

In this section, we develop the framework and notation in which
Theorem~\ref{BH-THM.CONVERGENCE1} will be proved. In
Section~\ref{BH-sec:special-case}, we will focus on establishing the
bulk of the estimates needed for the proof in the special case when
$W$ is an aligned cube. Then, in Section~\ref{BH-sec:general-case}, we
will use a covering argument to complete the proof for more general
windows.

Suppose that $\vL$ is a regular Euclidean model set, as described
above. Beginning from the definition in \eqref{BH-eq:FourBohrR}, we
make the simple observation that
\begin{equation}\label{BH-eq:IndicFnSum}
    \vol(B^{}_{\nts R})\,  a^{}_{\nts R}(t)\, 
     =\sum_{\lambda\in\vL^{}_{\nts R}}\! e(-t\lambda)
                   \, =\sum_{\lambda\in L}
                     e(-t\lambda)\,\bs{1}^{}_{B^{}_{\nts R}}(\lambda)
                     \,\bs{1}^{}_{W}(\lambda^{\nts\star}) \ts ,
\end{equation}
where summing over $\lambda\in L$ is the same as summing over
$(\lambda,\lambda^{\nts\star})\in\cL$, by
Remark~\ref{BH-rem:lattice-sum}.  We would like to apply the PSF to
this sum, but the summand does not satisfy the continuity or Fourier
transform properties required. Therefore, we replace the last sum by
\begin{equation}\label{BH-eq:F_RSum}
    \sum_{\lambda\in L}
    e(-t\lambda)\, F^{}_{\nts R}(\lambda,\lambda^{\nts\star}) \ts ,
\end{equation}
where the `mollified' function $F^{}_{\nts R}$ is defined by
\begin{equation}\label{BH-eq:F_RDef}
  F^{}_{\nts R}(\lambda,\lambda^{\nts\star}) \, = \,
  \frac{  \bigl(\bs{1}^{}_{B^{}_{\nts\nts R}}\!\ast
  \bs{1}^{}_{B^{}_{\nts\nts S_{\nts\nts R}}}\bigr)
  (\lambda)\cdot\bigl(\bs{1}^{}_{W}\ast
  \bs{1}^{}_{\nts B_{\epsilon^{}_{\nts\nts R}}}\bigr)
  (\text{$\lambda^{\nts\star}$})}
{2^k \, S^{d}_{\nts\nts R}\:\epsilon_{\nts\nts R}^{k-d}}\ts ,
\end{equation}
with $\{S^{}_{\nts R}\}$ and $\{\epsilon^{}_{\nts R}\}$ being
monotonic collections of positive real numbers chosen so that
\begin{equation}\label{BH-eq:S^{}_{R}Eps_RConds}
  S^{}_{\nts R} \, \to \, \infty,\quad S^{}_{\nts R}
  \, = \, o(R) \ts ,  \quad\text{and}\quad 
  \epsilon^{}_{\nts R} \, = \, o(1) \ts ,
  \quad \text{as}\quad R\to\infty \ts .
\end{equation}
In Proposition~\ref{BH-prop.F_RApprox}, we shall show that this does
not affect our final estimates. First, we establish the following
lemma, which is a consequence of a well-known uniform distribution
result for model sets; see \cite[Thm.~7.2]{BH-TAO1} and
\cite{BH-Bob02} for details and a general statement.

\begin{lemma}\label{BH-lem.NestedWindow}
  Suppose that\/ $(G,H,\cL)$ is a Euclidean CPS, as defined in
  Section~\textnormal{\ref{BH-sec:RegModSets}}. Let\/
  $(A_i)^{}_{i\in\NN}$ be a sequence of relatively compact, measurable
  subsets of\/ $H$, with non-empty interiors, boundaries of measure\/
  $0$, and with
\[
    \lim_{i\to\infty} \mathrm{vol}(A_i) \, = \, 0 \ts .
\]
Further, suppose that the sequence is nested, so\/
$A_{i+1}\subseteq A_i$ for all\/ $i\geqslant 1$. Then, there exists a
sequence\/ $ (R_i )^{}_{i\in\NN}$ of real numbers with the property
that, for all\/ $i\geqslant 1$ and all\/ $R\geqslant R_i$,
\[
   \card \bigl\{\lambda\in L :\lvert \lambda\rvert \le
   R~\text{and}~\lambda^{\nts\star}\nts\in A_i\bigr\}
   \, \leqslant \: 2^{d+1}\dens (\cL)\vol (A_i)R^d.
\]
Further, if\/ $(T_j )^{}_{j\in\NN}$ is any sequence of real numbers
tending to $\infty$, one has
\[
    \card \bigl\{\lambda\in L :\lvert \lambda\rvert \le
    T_j \text{ and } \lambda^{\nts\star}\nts\in A_j\bigr\}
    \, = \: o\bigl(
    T^{d}_j \bigr)\quad\text{as}\quad j\to\infty \ts .
\]
\end{lemma}

\begin{proof}
It is shown in \cite{BH-Schlottmann1998} that
\[
  \lim_{R\to\infty}\frac{\card \bigl\{\lambda\in L :
    \lvert \lambda\rvert \le
    R~\text{and}~\lambda^{\nts\star}\nts\in
    A_i\bigr\}}{(2R)^d} \, = \, \dens(\cL)\,
  \mathrm{vol}(A_i) \ts ,
\]
which, together with an obvious step involving the triangle inequality
that explains the extra factor of $2$, implies the first claim of the
lemma.

Next, let $ ( R_i )^{}_{i\in\NN}$ be a sequence of real numbers
satisfying the first claim. For each $i$, choose $j_i\geqslant i$
large enough so that $T_j\geqslant R_i$, for all $j\geqslant
j_i$. Then, using the fact that $A_j\subseteq A_i$ for all
$j\geqslant i,$ it follows that, for $j\geqslant j_i$,
\begin{align*}
   \card \bigl\{\lambda\in L :\lvert \lambda\rvert \le
  T^{}_j~\text{and}~\lambda^{\nts\star}\nts\in A_j\bigr\} \,
  & \leqslant \, \card \bigl\{\lambda\in L :\lvert \lambda\rvert \le
   T^{}_j~\text{and}~\lambda^{\nts\star}\nts\in A_i\bigr\} \\[1mm]
  & \leqslant \, 2^{d+1}\dens(\cL ) \, \mathrm{vol}(A_i) \, T_j^{d}.
\end{align*}
Taking the limit as $i\to\infty$ completes the argument.
\end{proof}

One can now connect the FB coefficients of the model set $\vL$ and of
its counterpart with the mollified strip and window as follows.

\begin{proposition}\label{BH-prop.F_RApprox}
  If\/ $\nts\vL$ is a regular model set, if\/ $\nts F^{}_{\nts R}$ is
  defined as in\/ \eqref{BH-eq:F_RDef}, and if\/ $\{S^{}_{\nts R}\}$
  and\/ $\{\epsilon^{}_{\nts R}\}$ satisfy\/
  \eqref{BH-eq:S^{}_{R}Eps_RConds}, one has
\begin{align*}
  \vol(B^{}_{\nts R})\, a^{}_{\nts R}(t) \, &=\sum_{\lambda\in L}
  e(-t\lambda)\,\bs{1}^{}_{B^{}_{\nts R}}(\lambda)
  \,\bs{1}^{}_{W}(\lambda^{\nts\star})\\[1mm]
  &=
  \sum_{\lambda\in L} e(-t\lambda)\, 
    F^{}_{\nts R}(\lambda,\lambda^{\nts\star})\; + \, o(R^d)
    \qquad \text{as } R\to\infty \ts .
\end{align*}
\end{proposition}

\begin{proof}
  The first equality in the conclusion of the proposition has already
  been established; see Eq.~\eqref{BH-eq:IndicFnSum}.  For $R>0$, let
  the functions $\varphi^{}_{\nts R}\colon G\rar\RR$ and
  $\psi^{}_{\nts R}\colon H\rar\RR$ be defined by
\begin{equation}\label{BH-eq:Phi_RPsi_RDef}
  \varphi^{\phantom{\chi}}_{\nts R}(x)\, = \, 
  \frac{\bigl( \bs{1}^{}_{B^{}_{\nts R}}\ast
    \bs{1}^{}_{B^{}_{\nts S_{\nts\nts R}}}\bigr)
    (x)}{2^d \, S^{d}_{\nts\nts R}}
  \quad\text{and}\quad\psi^{\phantom{\chi}}_{\nts R}(y) \, = \,
  \frac{\bigl(\bs{1}^{}_{W} \ast \bs{1}^{}_{\epsilon^{}_{\nts R}}\bigr)
  (y)}{2^{k-d}\, \epsilon_{\nts R}^{k-d}}\ts .
\end{equation}
For each $R$, the function
\[
    \varphi^{\phantom{\chi}}_{\nts R}-\bs{1}^{}_{B^{}_{\nts R}}
\]
is supported in the $S^{}_{\nts\nts R}\ts$-neighbourhood of the
boundary of $B^{}_{\nts R}$ in $G$ (physical space), which we denote
by $\partial(B^{}_{\nts R},S^{}_{\nts R})$. Similarly, the function
\[
     \psi^{\phantom{\chi}}_{\nts R} - \bs{1}^{}_{W}
\]
is supported in the $\epsilon^{}_{\nts R}$-neighbourhood of the
boundary of $W$ in $H$ (internal space), which we call
$\partial(W,\epsilon^{}_{\nts R})$.

The monotonicity of $\{\epsilon^{}_{\nts R}\}$ guarantees that, for any
sequence of numbers $R_1<R_2<\cdots$ tending to infinity, the collection
$\bigl\{\partial(W,\epsilon^{}_{\nts R_i})\bigr\}_{i\in\NN}$ is a nested
sequence of sets, and the assumption  $\vol(\partial W) =0$
implies that
\[
    \lim_{i\to\infty}\vol\bigl(\partial(W,\epsilon^{}_{\nts R_i})\bigr)
    \,= \, 0 \ts .
\]
Applying Lemma~\ref{BH-lem.NestedWindow}, we have
\[
  \card \bigl\{\lambda\in L :\lvert \lambda\rvert \le
  R^{}_i+S^{}_{\nts R_i}~\text{and}~\lambda^{\nts\star}\nts\in
  \partial(W,\epsilon^{}_{\nts R_i}) \bigr\}
  \, = \: o\bigl((R^{}_i+S^{}_{\nts R_i})^d\bigr)
\]
as $i\to\infty$. Consequently, as $R \to \infty$, we get
\begin{equation}\label{BH-eq:Eps_RIneq}
  \sum_{\lambda\in L}
  e(-t\lambda)\,\varphi^{\vphantom{\chi}}_{\nts R}(\lambda)\,
  \bs{1}^{}_{W}(\lambda^{\nts\star}) \, =
  \sum_{\lambda\in L} e(-t\lambda)\,
  F^{}_{\nts R}(\lambda,\lambda^{\nts\star})
  \; +\, o(R^d) \ts .
\end{equation}
Further, by the uniform distribution arguments outlined in the proof 
of \cite[Thm.~7.2]{BH-TAO1}, see also \cite{BH-Bob02}, we get
\[
  \card \bigl\{\lambda\in L :
  \lambda\in\partial(B^{}_{\nts R},S^{}_{\nts R} )~\text{and}~
  \lambda^{\nts\star}\nts\in W \bigr\}
  \, = \, O(R^{d-1}S_{\nts R}) \ts ,
\]
which is a measure of the thickened boundary.  Using our condition
$S^{}_{\nts R}=o(R)$, this implies
\begin{equation}\label{BH-eq:S^{}_{R}Ineq}
  \vol( B^{}_{\nts R}) \, a^{}_{\nts R}(t)\, =
  \sum_{\lambda\in L}
  e(-t\lambda)\,\varphi^{\vphantom{\chi}}_{\nts R}(\lambda)\,
  \bs{1}^{}_{W}(\lambda^{\nts\star})\; +\, o(R^d) \ts .
\end{equation}
Combining Eqs.~\eqref{BH-eq:Eps_RIneq} and
\eqref{BH-eq:S^{}_{R}Ineq} produces the desired conclusion.
\end{proof}

Returning to our main line of thought, we wish to apply the PSF to
the sum in Eq.~\eqref{BH-eq:F_RSum}. For $t\in G$, define
$\phi_t \colon G \times H \rar \CC$ by
\[
     \phi_t(x) \, = \, e(-t \ts \pi(x)) F_R(x ) \ts .
\]
This function is continuous with compact support, but we cannot, in
general, guarantee condition \eqref{BH-eq:PSF2} to be
satisfied. However, if we assume that $W$ is an \emph{aligned} cube in
$H$ with side length $\eta$, then, using Eqs.~\eqref{BH-eq:FourCoeffs1}
and \eqref{BH-eq:ConvFourTrans} and the notation of Eq.
\eqref{BH-eq:Phi_RPsi_RDef}, we have for
$(\theta,\theta^{\star})\in\cL^{*}$ that
\begin{equation}\label{BH-eq:FourCoeffProdEst}
\begin{split}
    \bigl|\widehat{\varphi^{\rule{0pt}{1.5pt}}_{\nts R}}
    (\theta+t)\bigr| \, \bigl|\widehat{\psi^{}_{\nts R}}
    (\theta^{\star}) \bigr| \, \leqslant \,
    & \prod_{i=1}^d \min\nts
     \left(2R\ts ,\myfrac{1}{| \theta_i+t_i| }\right)\,
     \min\nts\left(1,\myfrac{1}{S^{}_{\nts R}\,|
         \theta_i+t_i| }\right) \\[1mm]
   &\cdot\prod_{i=1}^{k-d}
	\min\nts\left(\eta,\myfrac{1}{| \theta_i^{\star}| }\right)\,
	\min\nts\left(1,\myfrac{1}{\epsilon^{}_{\nts R}\,
		| \theta_i^{\star}|}\right).
\end{split}
\end{equation}
Using this estimate, we see that
\[
  \sum_{\xi \in\cL^\ast_{\vphantom{t}}}
  \bigl| \widehat{\phi^{}_t}(\xi ) \bigr|
  \, =  \sum_{(\theta,\theta^{\star})\in\cL^{*}}
  \bigl| \widehat{\varphi^{\rule{0pt}{1.5pt}}_{\nts\nts R}}
  (\theta+t) \bigr| \,	\bigl| \widehat{\psi^{}_{\nts R}}
  (\theta^{\star} ) \bigr|  \, < \, \infty ,  \ts 
\]
by comparing the sum with the corresponding multiple integral. We omit
the details of this argument, as we will provide more precise bounds
for sums of this form in the next section. Since condition
\eqref{BH-eq:PSF2} is satisfied for this special choice of $W$ and,
applying Proposition~\ref{BH-prop.PSF} to $\phi^{}_t$ with
$y=0$,\index{Poisson~summation~formula} we get
\begin{equation}\label{BH-eq:PSFforAligned}
  \sum_{\lambda\in L} e(-t\lambda)\,
  F^{}_{\nts R}(\lambda,\lambda^{\nts\star})
  \, = \, \dens (\cL) \sum_{\substack{\theta\in L^{\nts\circledast}}}
  \widehat{\varphi^{\rule{0pt}{1.5pt}}_{\nts\nts R}}(\theta+t)\, 
  \widehat{\psi^{}_{\nts R}}(\theta^{\star} ) \ts .
\end{equation}
Note that summing over all $\theta\in L^{\circledast}=\pi (\cL^{*})$
is the same as summing over all lattice points
$(\theta,\theta^{\star})\in\cL^{*}$, due to
Remark~\ref{BH-rem:lattice-sum} and
Proposition~\ref{BH-prop.DualInj}. At this point, there are two
possibilities to consider. First, if $t\in L^{\circledast}$, the
contribution to the sum from $\theta=-t$ is
\begin{equation}\label{BH-eq:F_RSumMainTerms}
     \dens (\cL) \vol(B^{}_{\nts R}) \,
     \widehat{\psi^{}_{\nts R}}(-t^{\star})
     \, = \, \dens (\cL)  \vol( B^{}_{\nts R})\,
     \widehat{\bs{1}^{}_{W}}(-t^{\star})\,
     \bigl( 1+o(1) \bigr)
\end{equation}
as $R\to\infty$. This gives the main terms, which determine the
support of the diffraction measure. For the other terms, we will rely
on the upper bound from Eq.~\eqref{BH-eq:FourCoeffProdEst}.

Much of the remainder of the proof of
Theorem~\ref{BH-THM.CONVERGENCE1} is now centred around estimating
sums of products of the form given on the right-hand side of
Eq.~\eqref{BH-eq:FourCoeffProdEst}, in order to establish the result
we seek in the special case when $W$ is an aligned cube. Once we have
accomplished this, we will employ a covering argument to deal with the
case of more general windows.

\section{Proof of 
Theorem~{\protect\ref{BH-THM.CONVERGENCE1}}:  
Aligned cubes}\label{BH-sec:special-case}

Throughout this section, we will assume that $W$ is an aligned cube of
side length $\eta$ in $H$. Using the notation of the previous section,
for each $s\in\RR^d$, we define a set $\cA_s\subset \RR^k$ by
\[
  \cA_s \, = \, \bigl\{(x,y)\in\RR^d{\times}\ts \RR^{k-d} :
   ~ 0<\lvert x+s\rvert \leqslant S_{\nts R}^{-1}, 
  \lvert y\rvert \leqslant \epsilon_R^{-1}\bigr\} .  
\]
Our goal in this section is to establish the following result.

\begin{proposition}\label{BH-prop:AlignedWindEst}
  Suppose that\/ $W$ is an aligned cube in\/ $H$ of side length\/
  $\eta$, that\/ $t\in\RR^d$, $S_R\leqslant R$, and
  that
  \begin{equation}\label{BH-eq:AlCubeEstCond1}
  	\bigl\{\theta\in\RR^d:(\theta,0)\in\cL^*,~ 0<|\theta|\le
  S_R^{-1}~\text{or }~0<|\theta+t|\le
  S_R^{-1}\bigr\} \, = \, \varnothing \ts .
  \end{equation}
  Let\/ $\epsilon^{}_R$ be the infimum of the set of all real numbers
  for which
\begin{equation}\label{BH-eq:AlCubeEstCond2}
  \cA_0\cap\cL^* \ts = \, \cA_t\cap\cL^*
  \ts = \, \varnothing \ts .
\end{equation}
Then, if\/ $\epsilon_R\leqslant \eta$, we have that
\[
  \sum_{\substack{(\theta,\theta^{\star}\nts )\in\cL^*\nts\\
      \theta+t\ne 0}}\!\!  \widehat{\varphi^{\rule{0pt}{1.5pt}}_{\nts
      R}}(\theta+t) \, \widehat{\psi^{}_{\nts R}}(\theta^{\star})\ts
  \, \ll \, R^{d-1} S^{}_R \ts \eta^{k-d}
  + R^{d}\epsilon^{}_{\nts R} \ts \eta^{k-d-1}  ,
\]
where the implied constant depends only on\/ $k$ and\/ $d$.
\end{proposition}

Note that condition \eqref{BH-eq:AlCubeEstCond1} will be satisfied as
long as $S^{}_R$ is large enough. The role of this condition is to ensure
that the infimum defining $\epsilon^{}_{\nts R}$ exists. Furthermore,
for all $S^{}_R$ sufficiently large (depending on $t$), the condition
that $\epsilon^{}_{\nts R} \leqslant \eta$ will also be satisfied.

Proposition~\ref{BH-prop:AlignedWindEst} easily implies the statement
of Theorem~\ref{BH-THM.CONVERGENCE1} for windows which are aligned
cubes. However, we defer the details of this claim until the next
section, when we establish the theorem in its full generality.

In Section~\ref{BH-subsec:k=2,d=1}, we shall complete the proof of
Proposition~\ref{BH-prop:AlignedWindEst} for $k=2$ and $d=1$, followed
by the general case in Section~\ref{BH-subsec:k,d,general}. The proof
in higher dimensions is similar in structure, but is notationally more
complex and thus a little less transparent. We hope that this is
compensated by the explicit treatment of the special case.

\subsection{Special case ($k=2, d=1$)}\label{BH-subsec:k=2,d=1}

First, let us suppose that $k=2$ and $d=1$, that $W$ is an aligned
cube in $H$ (in this case, simply an interval) of side length $\eta$,
and that the hypotheses of Proposition~\ref{BH-prop:AlignedWindEst}
are all satisfied. Given $t\in\RR$, which is arbitrary but fixed, we
define subregions of $\RR^2=\RR \ts\ts {\times} \ts\ts \RR$ as
follows. First, set
\begin{align*}
  \cT^{}_{1,1}\, &= \, \{(x,y)\in\RR^2:\lvert x+t\rvert
                \leqslant S_{R}^{-1} \} \ts , \\[1mm]
  \cT^{}_{1,2}\, &= \, \{(x,y)\in\RR^2:\lvert x+t\rvert >
                S_{R}^{-1} \} \ts , \\[1mm]
  \cT_{1,1}^{\star} \, &= \, \{(x,y)\in\RR^2:\lvert y\rvert \leqslant 
       \epsilon_{\nts R}^{-1}\} \ts , \quad\text{and} \\[1mm]
  \cT_{1,2}^{\star} \, &= \, \{(x,y)\in\RR^2:\lvert y\rvert >
                   \epsilon_{\nts R}^{-1}\} \ts . 
\end{align*}
Also, for each $m,n\in\NN$, define
\begin{align*}
  \cB_{1,m}\, &= \, \{(x,y)\in\RR^2:m \ts S_{R}^{-1} <
                \lvert x^{}+t^{}\rvert \leqslant (m+1)S_{R}^{-1}\} \ts ,
                \quad\text{and}  \\[1mm]
  \cB_{1,n}^{\star}\, &= \, \{(x,y)\in\RR^2:n\ts\epsilon_{R}^{-1} <
                \lvert y\rvert \leqslant (n+1)\ts\epsilon_{R}^{-1}\} \ts ,
\end{align*}
so that $\cT^{}_{1,2}$ is the disjoint union of the sets $\cB_{1,m}$
and $\cT_{1,2}^{\star}$ is the disjoint union of the sets
$\cB_{1,n}^{\star}$.  Next, for each
$\sigma\in\{1,2\}\ts {\times} \ts \{1,2\}$, writing
$\sigma=(\sigma^{}_1,\sigma^{}_2)$, define
\[
\begin{split}  
   \cR^{}_\sigma \, &  = \, 
   \cT^{}_{1,\sigma^{}_1}\nts\cap\ts
   \cT_{1,\sigma^{}_2}^{\star}  \quad \text{and} \\[2mm]
    \varSigma_\sigma \, & = \!
    \sum_{\substack{(\theta,\theta^{\star}\nts )\in\cL^*\nts\cap
    \cR_\sigma\\ \theta+t\ne 0}}\!\!
    \left|\widehat{\varphi^{\rule{0pt}{1.5pt}}_{\nts R}}(\theta+t) \, 
    \widehat{\psi^{}_{\nts R}}(\theta^{\star})\right|\ts .
\end{split}
\]
This notation may seem overly complicated for the case at hand, but it
will allow for easier generalisation to higher dimensions in the next
section.

Now, for fixed $t$, we will consider how to bound $\varSigma_\sigma$,
for each of the four different choices of $\sigma$. First, since
 $\cA_t\cap \cL^*=\varnothing$, we have that $\varSigma_{(1,1)}=0$.

To bound $\varSigma^{}_{(2,2)}$, write
\[
  \cR^{}_{(2,2)} \, = \bigcup_{m,n\in\NN}\cB^{}_{1,m}\cap\cB_{1,n}^\star,
\]
and note that, by the fact that $\cA_0\cap \cL^*=\varnothing$, each of
the sets $\cB^{}_{1,m}\cap\cB_{1,n}^\star$ contains at most 4 lattice
points{\ts}\footnote{To see this more clearly, it may also be helpful
  to observe that, since $\pi|_\cL^*$ is injective
  (Proposition~\ref{BH-prop.DualInj}), the only lattice point of
  $\cL^*$ in the closure of $\cA_0$ is the point $0$.} of
$\cL^\star$. Thus, with the bounds from
\eqref{BH-eq:FourCoeffProdEst},
\begin{equation}\begin{split}
  \varSigma^{}_{(2,2)}\, & \leqslant
  \sum_{\substack{(\theta,\theta^{\star}\nts )\in
   \cL_{\vphantom{t}}^*\nts\cap\cR_{(2,2)}}}
   \myfrac{1}{S^{}_{\nts R} \, \lvert \theta+t\rvert^2}\,
    \myfrac{1}{\epsilon^{}_{\nts R}\,
   \lvert \theta^{\star}\rvert^2}\label{BH-eq:Err1}\\[1mm]
 &\leqslant \, \myfrac{4}{S^{}_{\nts R} \ts \epsilon^{}_{\nts R}}
    \sum_{m,n\in\NN} \frac{1}{(m S_R^{-1})^2(n\ts \epsilon_R^{-1})^2}
    \, \ll \,  S^{}_{\nts R}\ts\ts \epsilon^{}_{\nts R} \ts .
      \end{split}
\end{equation}
Similarly, since
\[
  \cR^{}_{(2,1)} \, = \bigcup_{m\in\NN}
  \cB^{}_{1,m}\cap\cT_{1,1}^\star \ts ,
\]
we find that
\begin{equation}\label{BH-eq:Err2}
  \varSigma^{}_{(2,1)}\,  \leqslant
  \sum_{\substack{(\theta,\theta^{\star}\nts )\in\cL_{\vphantom{t}}^*\nts
	\cap\ts\cR_{(2,1)}}}\myfrac{\eta}{S^{}_{\nts R}\,\lvert
	\theta+t\rvert^2} 
      \, \leqslant \, \myfrac{4\ts \eta}{S^{}_{\nts R}}\sum_{m\in\NN}
      \frac{1}{(m\ts S_R^{-1})^2} \, \ll \, S^{}_{\nts R}\ts\ts\eta \ts .
\end{equation}

Finally, writing
\[
  \varSigma^{}_{(1,2)} \, = \bigcup_{n\in\NN}\cT^{}_{1,1}
  \cap\cB_{1,n}^{\star} \ts ,
\]
we obtain
\begin{equation}\label{BH-eq:Err4}
\varSigma^{}_{(1,2)} \, \leqslant
  \!\sum_{\substack{(\theta,\theta^{\star}\nts )\in
   \cL_{\vphantom{t}}^*\nts \cap \ts \cR_{(1,2)}}}\!
   \myfrac{2R}{\epsilon^{}_{\nts R}\, \lvert
   \theta^{\star}\nts\rvert^2} 
   \, \leqslant \,\myfrac{8R}{\epsilon^{}_{\nts R}}
   \sum_{n\in\NN}\myfrac{1}{(n\ts\epsilon^{-1}_R)^2}\,
   \ll \, R \ts \epsilon^{}_{\nts R}  \ts.
\end{equation}
Combining the estimates in Eqs.~\eqref{BH-eq:Err1}--\eqref{BH-eq:Err4}
establishes Proposition~\ref{BH-prop:AlignedWindEst} for $k=2$ and
$d=1$.

\subsection{Higher dimensions}\label{BH-subsec:k,d,general}

The proof of Proposition~\ref{BH-prop:AlignedWindEst}, for general $k$
and $d$, parallels the structure of the proof in the previous
section. To this end, define subregions of the product space
$\RR^k=\RR^d\ts\ts {\times} \ts\ts \RR^{k-d}$ as follows. For
$1\leqslant i\leqslant d$, let
\begin{align*}
  \cT^{}_{i,1} \, & = \, \{(x,y)\in\RR^d\nts\times\RR^{k-d}:
       \lvert x^{}_i+t^{}_i\rvert\leqslant S_{R}^{-1} \} \ts , \\[1mm]
  \cT^{}_{i,2} \, & = \, \{(x,y)\in\RR^d\nts\times\RR^{k-d}:
                    \lvert x^{}_i+t^{}_i\rvert > S_{R}^{-1}\} \ts ,
\end{align*}
and, for $1\leqslant i\leqslant k-d$, let
\begin{align*}
  \cT_{i,1}^{\star} \, & = \, \{(x,y)\in\RR^d\nts\times\RR^{k-d}:
        \lvert y^{}_i\rvert \leqslant 
	\epsilon_{\nts R}^{-1}\} \ts , \quad\text{and} \\[1mm]
  \cT_{i,2}^{\star} \, & = \, \{(x,y)\in\RR^d\nts\times\RR^{k-d}:
        \lvert y^{}_{i}\rvert > 
	\epsilon_{\nts R}^{-1}\} \ts .
\end{align*}
For $1\leqslant i\leqslant d$ and $m\in\NN$, set
\[
   \cB^{}_{i,m} \, = \,
   \{(x,y)\in\RR^d\nts\times\RR^{k-d}:mS_{R}^{-1}
   <\lvert x^{}_i+t^{}_i\rvert \leqslant (m+1)S_{R}^{-1}\} \ts ,
\]
so that $\cT^{}_{i,2}$ is the disjoint union of the sets
$\cB^{}_{i,m}$. Similarly, for each index $1\leqslant i\leqslant k-d$
and every $n\in\NN$, set
\[
   \cB^{\star}_{i,n} \, = \,
   \{(x,y)\in\RR^d\nts\times\RR^{k-d}:n\epsilon_{R}^{-1}
   < \lvert y^{}_i\rvert \leqslant (n+1)\ts \epsilon_{R}^{-1}\} \ts ,
\]
so that $\cT^{\star}_{i,2}$ is the disjoint union of the sets
$\cB^{\star}_{i,n}$.

For each $\sigma\in\{1,2\}^k$, with
$\sigma=(\sigma^{}_1,\ldots ,\sigma^{}_k)$, we define
\[
  \cR^{}_\sigma \,  = \, \bigcap_{i=1}^d\cT^{}_{i,\sigma^{}_{\nts\nts i}}
    \, \cap\, \bigcap_{j=1}^{k-d}\cT_{j,\sigma^{}_{\nts\nts j+d}}^{\star}
             \quad \text{and} \quad
  \varSigma_\sigma \,  =\!\sum_{\substack{(\theta,\theta^{\star}\nts )
  \in\cL_{\vphantom{t}}^*\nts \cap \ts \cR_\sigma \\ \theta+t\ne 0}}\!\!
  \left|\widehat{\varphi^{\rule{0pt}{1.5pt}}_{\nts R}}(\theta+t)\,
  \widehat{\psi^{}_{\nts R}}(\theta^{\star}\nts )\right|\ts .
\]
Given such a $\sigma$, for $\ell\in\{1,2\}$, we set
\[
  j^{}_\ell\,  = \, \card\{1\leqslant i\leqslant d:
     \sigma^{}_{\nts i}=\ell\} \quad \text{and} \quad
  j_\ell^{\star}\,  =\, \card\{d+1\leqslant i\leqslant k:
  \sigma^{}_{\nts i}=\ell\} \ts .
\]
Condition~\eqref{BH-eq:AlCubeEstCond2} in the statement of
Proposition~\ref{BH-prop:AlignedWindEst} guarantees
$\varSigma_\sigma=0$ whenever $j^{}_2=j_2^{\star}=0$. Therefore,
suppose that $j_2>0$ or $j_2^\star>0$ and write
\[
\begin{split}  
     \{\nu^{}_{1},\ldots,\nu^{}_{j^{}_{2}}\}
     \, & = \, \{1\leqslant i\leqslant d : \sigma^{}_{i} = 2 \}
     \, \subseteq \, \{1,\ldots ,d\} \quad \text{and} \\[2mm]
    \{\rho^{}_{1},\ldots,\rho^{}_{j^{\star}_{2}}\}
    \,& = \, \{d+1\leqslant i\leqslant k : \sigma^{}_{i} = 2 \}
    \, \subseteq \, \{d+1,\ldots ,k\} \ts .
\end{split}
\]
If $j_2=0$ or $j_2^\star=0$, one of these sets could be empty, which
should be interpreted appropriately in the equations to follow (empty
products are assumed to equal 1). The estimate from
\eqref{BH-eq:FourCoeffProdEst} implies for
$(\theta,\theta^{\star}\nts )\in\cL^*\nts\cap\cR^{}_{\sigma}$ that
\begin{equation}\label{BH-eq:FourCoeffProdEst2}
 \left|\widehat{\varphi^{\rule{0pt}{1.5pt}}_{\nts R}}(\theta+t)\,
  \widehat{\psi^{}_{\nts R}}(\theta^{\star}\nts )\right| \,\leqslant \,
  \myfrac{(2R_{\vphantom{t}})^{\ts j^{}_1}
	\eta^{j_1^{\star}}}{S_{\nts R}^{j^{}_2}\epsilon^{j_2^\star}_{\nts R}}
    \prod_{\substack{i=1\\\sigma^{}_{\nts i}=2}}^d
    \myfrac{1}{\lvert \theta^{}_i+t^{}_i\rvert^2}
    \prod_{\substack{i=1\\\sigma^{}_{\nts d+i}=2}}^{k-d}
    \myfrac{1}{\lvert 
	\theta_i^{\star}\nts\rvert^2}\ts.
\end{equation}
The set $\cR^{}_{\sigma}$ is a disjoint union over all $m\in\NN^{j_2}$
and $n\in\NN^{j_2^\star}$ of the sets
\[
  \bigcap_{\substack{i=1 \\ \sigma_i=1}}^d\cT_{i,1} \, \cap \,
  \bigcap_{i=1}^{j_2}\cB_{\nu_i,m_i} \, \cap \!
  \bigcap_{\substack{i=1 \\
      \sigma_{i+d}=1}}^{k-d}\cT^\star_{i,\sigma_{i+d}} \, \cap \,
  \bigcap_{i=1}^{j^\star_2}\cB^\star_{\rho_i,n_i} \ts ,
\]
and, since $\cA_0\cap\cL^*=\varnothing$, each of these sets
contains at most $2^k$ points of $\cL^*$. It follows from this in
conjunction with Eq.~\eqref{BH-eq:FourCoeffProdEst2} that
\[
  \varSigma_\sigma \,
  \leqslant \myfrac{2^k(2R_{\vphantom{t}})^{\ts j^{}_1}
    \eta^{j_1^{\star}}}{S_{\nts R}^{j^{}_2}\epsilon^{j_2^\star}_{\nts R}}
    \sum_{m\in\NN^{j_2}}\sum_{n\in\NN^{j_2^\star}} \prod_{i=1}^{j_2}
    \frac{1}{\bigl( m^{}_i S_R^{-1} \bigr)^2}\prod_{i=1}^{j_2^\star}
    \frac{1}{\bigl( n^{}_i \epsilon_R^{-1} \bigr)^2}
    \, \ll \, R^{d-j_2}S_R^{j_2}\epsilon_R^{j_2^\star}\eta^{k-d-j_2^\star} \ts,
\]
where, in the last step, we have also used the facts that
$j_1+j_2=d$ and $j_1^\star+j_2^\star=k-d$. The largest
potential error terms arise from the cases when $(j_2,j_2^\star)$ is
$(1,0)$ or $(0,1)$, which completes the proof of
Proposition~\ref{BH-prop:AlignedWindEst}.

\section{Proof of 
Theorem~\protect\ref{BH-THM.CONVERGENCE1}: General
case}\label{BH-sec:general-case}

Let us now suppose that $W$ is any window for which $\oplam(W)$ is a
regular model set. For each $n\in\NN$, let
$\bigl\{\cD_1^{(n)},\cD_2^{(n)},\ldots ,\cD_{M(n)}^{(n)}\bigr\}$ be
the collection of dyadic cubes of the form
\[
  \prod_{j=1}^{k-d} \left[\frac{i_j}{2^n},\frac{i_j+1}{2^n} \right] \,
  \subseteq \, H, \qquad \text{with }\ts i_1,\ldots ,i_{k-d}\in\ZZ \ts ,
\]
which intersect $\partial W$. Similarly, let
$\bigl\{\cI_1^{(n)},\cI_2^{(n)},\ldots ,\cI_{N(n)}^{(n)}\bigr\}$ be
the collection of dyadic cubes of the above form which lie in the
interior of $W\!$, and let
\[
  \cD^{(n)} \, = \bigcup_{i=1}^{M(n)}\cD_i^{(n)}
  \quad\text{and}\quad
  \cI^{(n)} \, =\bigcup_{i=1}^{N(n)}\cI_i^{(n)}.
\]
Since $W$ has almost no boundary by assumption, we have
$\vol (\cD^{(n)}) \xrightarrow{\ts n\to\infty \ts} 0$ and
\begin{equation}\label{BH-eq:I-bound}
   \vol (\cI^{(n)}) \, = \, \frac{N(n)}{2^{n(k-d)}}
   \, \leqslant \, \vol (W) \ts .
\end{equation}

From Eq.~\eqref{BH-eq:IndicFnSum}, we have that
\begin{align}
  \vol & (B^{}_{\nts R})\,  a^{}_{\nts R}(t)\, =\sum_{\lambda\in L}
  e(-t\lambda)\,\bs{1}^{}_{B^{}_{\nts R}}(\lambda)
  \,\bs{1}^{}_{W}(\lambda^{\nts\star}) \ts \nonumber\\
  & = \sum_{\lambda\in L} e(-t\lambda)\,\bs{1}^{}_{B^{}_{\nts R}}(\lambda)
    \,\bs{1}^{}_{\cI^{(n)}}(\lambda^{\nts\star}) +
    O \biggl( \, \sum_{\lambda\in L} \bs{1}^{}_{B^{}_{\nts R}}(\lambda)
    \,\bs{1}^{}_{W\setminus\cI^{(n)}}(\lambda^{\nts\star})
    \biggr) \nonumber\\[1mm]
  & = \sum_{i=1}^{N(n)} \biggl( \, \sum_{\lambda\in L}
    e(-t\lambda)\,\bs{1}^{}_{B^{}_{\nts R}}(\lambda)
    \,\bs{1}^{}_{\cI_i^{(n)}}(\lambda^{\nts\star})\biggr) +
    O \biggl(\, \sum_{\lambda\in L} \bs{1}^{}_{B^{}_{\nts R}}(\lambda)
    \,\bs{1}^{}_{\cD^{(n)}}(\lambda^{\nts\star})\biggr).
    \label{BH-eq:CubeSlice1}
\end{align}
We now focus on each of the inner sums in the first term of
Eq.~\eqref{BH-eq:CubeSlice1}, with a view towards using
Propositions~\ref{BH-prop.F_RApprox} and
\ref{BH-prop:AlignedWindEst}. To avoid notational ambiguity, for each
choice of $R$, $S^{}_R$, $\epsilon^{}_R$, $n$ and $i$, we let
$F_R^{(i,n)}$ and $\psi_R^{(i,n)}$ denote the functions from the proof
of Proposition~\ref{BH-prop.F_RApprox} that correspond to the window
$\cI_i^{(n)}$ (note that the function $\varphi^{}_R$ does not depend
on the choice of window). Since $\cI_i^{(n)}$ is an aligned cube, we
have from Eq.~\eqref{BH-eq:PSFforAligned} that
\begin{equation}\label{BH-eq:PSFforAligned2}
  \sum_{\lambda\in L} e(-t\lambda)\,
  F^{(i,n)}_{\nts R}(\lambda,\lambda^{\nts\star})
  \, = \, \dens (\cL) \sum_{\substack{\theta\in L^{\nts\circledast}}}
  \widehat{\varphi^{\rule{0pt}{1.5pt}}_{\nts\nts R}}(\theta+t)\, 
\widehat{\psi^{(i,n)}_{\nts R}}(\theta^{\star} ) \ts . 
\end{equation}
We now divide our analysis into two cases.
\smallskip

\noindent \textbf{Case 1 ($t\notin L^{\circledast}$):} In this case,
there is no contribution to a `main term' in the above sum, and we may
apply Proposition~\ref{BH-prop:AlignedWindEst} directly. Let $C_n>0$
be large enough so that
\[
  \big\{\theta\in\RR^d: (\theta,0)\in\cL^*,~ 0 < |\theta| \le
  C_n^{-1/2}~\text{or }~0 < |\theta+t| \le
  C_n^{-1/2} \big\} \, = \, \varnothing \ts ,
\]
and also so that
\[
   C_n^{-1/2} \, < \, \min\big\{|\theta+t|>0 :
   (\theta,\theta^{\star})\in\cL^*, |\theta^{\star}|
   \leqslant 2^{n}\big\} ,
\]
and
\[
   C_n^{-1/2} \, < \, \min\big\{|\theta|>0 :
   (\theta,\theta^{\star})\in\cL^*,
   |\theta^{\star}|\leqslant 2^{n}\big\} .
\]
Then, for each $R\geqslant C^{}_n$, set $S^{}_{\nts R}=\sqrt{R\ts}$. The
requirements on $C_n$ guarantee that condition
\eqref{BH-eq:AlCubeEstCond1} in the statement of the proposition is
satisfied, and also that the quantity defined as an infimum in the
proposition, which we will label as $\epsilon^{}_{\nts R,n}$, is at most
$2^{-n}$ (i.e. so that the bound in the conclusion of the proposition
will hold when applied with $\eta=2^{-n}$).

It is clear that the numbers $C_n$ may be chosen as above so that
$C_n$ tends monotonically to $\infty$ as $n\to\infty$. This also
guarantees that $\epsilon^{}_{\nts R,n}$ tends monotonically to $0$ as
$n\to\infty$.

Applying Proposition~\ref{BH-prop:AlignedWindEst}, together with
Eq.~\eqref{BH-eq:PSFforAligned2}, we have that, whenever
$R\geqslant C_n$,
\begin{align}
  \sum_{i=1}^{N(n)} \biggl( \,
  & \sum_{\lambda\in L} e(-t\lambda)\, F^{(i,n)}_{\nts R}
    (\lambda,\lambda^{\nts\star})\biggr) \nonumber \\[1mm]
  & \ll \, N(n)\bigl( R^{d-1/2} 2^{-n(k-d)} +R^d \ts
    \epsilon^{}_{\nts R,n}2^{-n(k-d-1)} \bigr) \nonumber \\[2mm]
  & \ll \, R^d \bigl( R^{-1/2} + 2^n \ts \epsilon^{}_{\nts R,n}
    \bigr),  \label{BH-eq:CubeSlice2}
\end{align}
where Eq.~\eqref{BH-eq:I-bound} was used for the last line.  Now,
choose $R_n\geqslant C_n$ in such a way that, for all
$R\geqslant R_n$, the following three conditions are satisfied,
\begin{align*}
\text{(i)}&\hspace*{10bp}
      R^{-1/2} + 2^n\epsilon^{}_{\nts R,n} \, \leqslant \,
      \vol (\cD^{(n)}) \ts , \\[1mm]
\text{(ii)}&\hspace*{10bp}
    \sum_{\lambda\in L} \bs{1}^{}_{B^{}_{\nts R}}(\lambda)
    \,\bs{1}^{}_{\cD^{(n)}}(\lambda^{\nts\star}) \, \leqslant \,
    2^{d+1}\dens(\cL)\vol (\cD^{(n)})R^d, \quad\text{ and} \\[1mm]
\text{(iii)}& \hspace*{10bp}
              \text{for each } 1\leqslant i\leqslant N(n),
              \text{ one has the estimate} \\
&\biggl|\sum_{\lambda\in L} e(-t\lambda)\,\bs{1}^{}_{B^{}_{\nts R}}(\lambda)
  \,\bs{1}^{}_{\cI_i^{(n)}}(\lambda^{\nts\star}) \, -\sum_{\lambda\in L}
  e(-t\lambda)\, F^{(i,n)}_{\nts R}(\lambda,\lambda^{\nts\star})
   \biggr| \, \leqslant \, \frac{\vol (\cD^{(n)})R^d}{2^{n(k-d)}} \ts .
\end{align*}
Choosing $R_n$ in this way is possible due to the fact that
$\epsilon^{}_{\nts R,n}$ tends to $0$ as $R\to\infty$ (for (i)),
Lemma~\ref{BH-lem.NestedWindow} (for (ii)), and
Proposition~\ref{BH-prop.F_RApprox} (for (iii)). Combining
Eqs.~\eqref{BH-eq:CubeSlice1} and \eqref{BH-eq:CubeSlice2} we get
that, for all $n\in\NN$ and for all $R\geqslant R_n$, one has
\begin{align*}
  \vol (B^{}_R)\, a^{}_R (t)  &=\sum_{i=1}^{N(n)}
         \biggl(\, \sum_{\lambda\in L}
    e(-t\lambda)\, F^{(i,n)}_{\nts R}(\lambda,\lambda^{\nts\star})
    \biggr) + O \bigl( \vol(\cD^{(n)})R^d\bigr) \\[1mm]
    &\ll \, \vol(\cD^{(n)})R^d \ts ,
\end{align*}
which clearly gives $a^{}_{R} (t) \ll \vol(\cD^{(n)})$.  Finally,
letting $n\to\infty$ (which also implies $R_n\to \infty$) completes
the proof of Theorem~\ref{BH-THM.CONVERGENCE1} in this case.
\smallskip

\noindent \textbf{Case 2 ($t\in L^{\circledast}$):}
Here, the analysis is essentially the same as in Case 1,
except that there is one more term in the sum on the RHS of
Eq.~\eqref{BH-eq:PSFforAligned2} (corresponding to $\theta=-t$), which
is not accounted for by Proposition~\ref{BH-prop:AlignedWindEst}. At
this point it should be clear that, because of Proposition
\ref{BH-prop.DualInj}, there can be at most one such term. Taking this
into account, we have as before that, for $n\in\NN$ and
$R\geqslant R_n$,
\[
  \vol (B^{}_R) \, a^{}_R(t) \, = \,
  \dens(\cL) \, \widehat{\varphi^{}_R}(0)
  \sum_{i=1}^{N(n)}\widehat{\psi_R^{(i,n)}}(-t^{\star})
  \, + \, O \bigl( \vol(\cD^{(n)})R^d \bigr).
\]
Note that $\widehat{\varphi^{}_R}(0) = \vol (B^{}_R)$,
together with
\begin{align}
  \lim_{R\to\infty}\widehat{\psi_R^{(i,n)}}(-t^{\star}) \,
  & = \, \widehat{\bs{1}^{}_{\cI_i^{(n)}}}(-t^{\nts\star}) \ts ,
    \quad\text{ and} \nonumber \\[1mm]
  \lim_{n\to\infty}\sum_{i=1}^{N(n)}\widehat{\bs{1}^{}_{\cI_i^{(n)}}}
  (-t^{\star}) \, & = \lim_{n\to\infty}\widehat{\bs{1}^{}_{\cI^{(n)}}}
  (-t^{\star}) \, = \, \widehat{\bs{1}^{}_{W}}(-t^{\nts\star}) \ts .
   \label{BH-eq:CubeSlice3}
\end{align}
Thus, possibly after increasing $R_n$, we can ensure that, for
each $1\leqslant i\leqslant N(n)$,
\[
  \biggl| \widehat{\psi_R^{(i,n)}}(-t^{\star}) -
  \widehat{\bs{1}^{}_{\cI_i^{(n)}}}(-t^{\nts\star})\biggr|
  \, \leqslant \, \frac{\vol(\cD^{(n)})}{2^{n(k-d)}} \ts .
\]
Then, for all $n\in\NN$ and $R\geqslant R_n$, we can use the triangle
inequality in conjunction with Eq.~\eqref{BH-eq:I-bound} to obtain
\[
  a^{}_R (t) \, = \, \dens(\cL) \,
  \widehat{\bs{1}^{}_{\cI^{(n)}}}(-t^{\star})
  \, + \, O \bigl( \vol(\cD^{(n)}) \bigr).
\]
By Eq.~\eqref{BH-eq:CubeSlice3}, this implies 
\[
  \lim_{R\to\infty} a^{}_R (t) \, = \,
  \dens(\cL) \, \widehat{\bs{1}^{}_{W}}(-t^{\nts\star}) \ts ,
\]
which completes our argument.

\bigskip
\bigskip

\section*{Acknowledgements}

It is our pleasure to thank Philipp Gohlke and Neil Ma\~{n}ibo for
discussions and suggestions, and Uwe Grimm for comments and support
with the manuscript in its early stage.  This work was supported by
the German Research Foundation (DFG, Deutsche Forschungsgemeinschaft)
under grant SFB 1283/2 2021 -- 317210226, and by the National
Science Foundation (NSF) under grant DMS 2001248.


\begin{thebibliography}{AO2}\itemsep=2pt

\bibitem{BH-TAO1}
Baake M.\ and Grimm U.\ (2013).
\textit{Aperiodic Order. Vol.~1: A Mathematical Invitation}
(Cambridge University Press, Cambridge).

\bibitem{BH-TAO2}
Baake M.\ and Grimm U.\ (eds.) (2017).
\textit{Aperiodic Order. Vol.~2: Crystallography and Almost Periodicity}
(Cambridge Uni\-ver\-si\-ty Press, Cambridge).

\bibitem{BH-Cassels}
Cassels, J.W.S.\ (1971).
\textit{An Introduction to the Geometry of Numbers}
(Springer, Berlin).

\bibitem{BH-Hof}
Hof A.\ (1995).
On diffraction by aperiodic structures,
\textit{Commun.\ Math.\ Phys.} \textbf{169}, 25--43.

\bibitem{BH-Daniel}
Lenz D.\ (2009).
Continuity of eigenfunctions of uniquely ergodic dynamical 
systems and intensity of Bragg peaks,
\textit{Commun.\ Math.\ Phys.} \textbf{287}, 225--258.
\newline
\texttt{arXiv:math-ph/0608026}.

\bibitem{BH-Nato}
Moody R.V.\ (ed.) (1997).
\textit{The Mathematics of Long-Range Aperiodic Order},
NATO ASI Series C 489 
(Kluwer, Dordrecht).

\bibitem{BH-Bob-old}
Moody R.V.\ (1997).
Meyer sets and their duals.
In~\cite{BH-Nato}, pp.\ 403--441.

\bibitem{BH-Bob}
Moody R.V.\ (2000).
Model sets:\ A survey.
In \textit{From Quasicrystals to More Complex Systems},
Axel F., D\'enoyer F.\ and Gazeau J.P.\ (eds.), 
pp.\ 145--166
(EDP Sciences, Les Ulis, and Springer, Berlin).
\texttt{arXiv:math.MG/0002020}.

\bibitem{BH-Bob02}
Moody R.V.\ (2002).
Uniform distribution in model sets,
\textit{Can.\ Math.\ Bull.} \textbf{45}, 123--130.  

\bibitem{BH-Schlottmann1998}
Schlottmann M.\ (1998).
Cut-and-project sets in locally compact Abelian groups.
\textit{Quasicrystals and Discrete Geometry},
Fields Institute Monographs, vol.~10,
J.~Patera (ed.), pp.~247--264, 
(Amer.\ Math.\ Soc., Providence, RI). 

\bibitem{BH-SteiWeis}
Stein E.M.\ and Weiss G.\ (1971).
\textit{Introduction to Fourier Analysis on Euclidean Spaces}
(Princeton University Press, Princeton).

\end{thebibliography}
\end{document}